\documentclass{birkjour}
\usepackage{amsmath}
\usepackage{amssymb,dsfont}
\usepackage{amsfonts}
\usepackage{amsthm}
\usepackage{color}

\def\dim{\mathop{\rm dim}\nolimits}

\def\esssup{\mathop{\rm ess\,sup}\nolimits}
\def\e{\mathop{\rm e}\nolimits}
\def\spn{\mathop{\rm span}\nolimits}
\def\D{\,\mathrm d}

\def\Z{\mathbb Z}

\newcommand{\N}{\mathbb{N}}

\newcommand{\ie}{i.e.}

\newtheorem{theorem}{Theorem}[section]
\newtheorem{remark}{Remark}[section]

\numberwithin{equation}{section}
\title[Shift-invariant subspaces of Sobolev type]{Shift-invariant subspaces of Sobolev type}
\author[A. Aksentijevi\' c]{Aleksandar Aksentijevi\' c}
\address{Department of Mathematics and Informatics,
Faculty of Science,
University of Kragujevac,
Radoja Domanovi\'ca 12,
34000  Kragujevac,
Serbia}
\email{aksentijevic@kg.ac.rs}

\author[S. Aleksi\' c]{Suzana Aleksi\' c}
\address{Department of Mathematics and Informatics,
Faculty of Science,
University of Kragujevac,
Radoja Domanovi\'ca 12,
34000  Kragujevac,
Serbia}
\email{suzana.aleksic@pmf.kg.ac.rs}

\begin{document}
\date{}
\maketitle

\noindent
{\bf Abstract.} This paper has the characteristics of a review paper in which results of shift-invariant subspaces of Sobolev type are summarized without proofs.  The structure of shift-invariant spaces $V_s$, $s\in\mathbb{R}$,  generated by at most countable family of generators, which are subspaces  of Sobolev spaces  $H^s(\mathbb{R}^n)$, are announced in \cite{aap} and Bessel sequences, frames and Riesz families of such spaces are characterized. With the Fourier multiplier $\left(1-\frac{\Delta}{4\pi^2}\right)^{s/2}f=\mathcal{F}^{-1}\big((1+|t|^2)^{s/2}\widehat{f}(t)\big)$, we are able to extend notions and theorems in \cite{MB} to spaces of the Sobolev type.

\medskip
\noindent
{\bf Key Words and Phrases}: shift-invariant space; Bessel sequence; frame; Riesz family; range operator; shift-preserving operator; dual frame.
\medskip

\noindent
{\bf 2000 Mathematics Subject Classification}: 42C15, 42B30

 \section{Introduction}
\indent
Ron and Shen in \cite{RS} gave the characterization of frames and Riesz
families of shift-invariant systems using the Gramian and the dual Gramian matrices. Involving the range function approach, Bownik in \cite{MB} derived a simplified approach to the analysis of frames and Riesz families. The characterization of frames and Riesz families of the form $E(\mathcal{A}_I)=\{f(\cdot-k):f\in\mathcal{A}_I,\, k\in\Z^n,\, \mathcal{A}_I\subset L^2\}$, where $I\subseteq\N$, for shift-invariant subspaces of $L^2(\mathbb{R}^n)$ is given in terms of their behavior on subspaces of $\ell^2(\mathbb{Z}^n)$ and the range function.

The results of Bownik and Aguilera, papers \cite{MB} and \cite{ACCP1, ACCP2}, were generalized to shift-invariant spaces of the Sobolev type (see \cite{aap}, \cite{aap1}).

We will denote by $\mathbb{N}$, $\mathbb{Z}$, $\mathbb{R}$ the sets of positive integers, integers, real
numbers, respectively; $\mathbb{N}_0=\mathbb{N}\cup\{0\}$, $\mathbb{N}^n_0=(\mathbb{N}\cup\{0\})^n$, $\mathbb{T}^n=\left[-\frac{1}{2},\frac{1}{2}\right)^n$ where $n\in\mathbb{N}$. The standard multi-index notation will be used: $|x|=\sqrt{x_1^2+\cdots+x_n^2}$, $\langle x,y\rangle=\sum_{k=1}^nx_ky_k$, $x,y\in\mathbb{R}^n$.

A function $f$ belongs to $L^p_s=L^p_s(\mathbb{R}^n)$, $p\geqslant1$, with weight function $\mu_s(\cdot)=(1+|\cdot|^2)^{s/2}$
if $f\mu_s$ belongs to $L^p=L^p(\mathbb{R}^n)$. Equipped with the norm $\|f\|_{L^p_s}=\|f\mu_s\|_{L^p}$, the space $L^p_s$
is a
Banach space and the inner product is given in the usual way. For $s=0$,
the usual $L^p$ space is obtained. We also consider the weighted sequence spaces $\ell^2_s=\ell^2_s(\mathbb{Z}^n)$: a sequence
$c=(c_k)_{k\in\mathbb{Z}^n}$ belongs to $\ell^2_s$ if $(c\mu_s)_{k\in\mathbb{Z}^n}=\left(c_k\mu_s(k)\right)_{k\in\mathbb{Z}^n}$ belongs to $\ell^2$ with norm $\|c\|_{\ell^2_s}=\|c\mu_s\|_{\ell^2}$, where $(\mu_s(k))_{k\in\mathbb{Z}^n}$ is the restriction of $\mu_s$ to $\mathbb{Z}^n$. Equipped with the inner product
    $\langle (c_k)_{k\in\mathbb{Z}^n},(d_k)_{k\in\mathbb{Z}^n}\rangle_{\ell_s^2}=\sum_{k\in\mathbb{Z}^n}c_k\overline{d}_k\mu_s^2(k)$, the space $\ell_s^2$ is a Hilbert space.

The Fourier transform and the inverse Fourier transform of an integrable function $f$ are defined by
$\mathcal{F}f(t)=\widehat{f}(t)=\int_{\mathbb{R}^n}f(x)\e^{-2\pi \sqrt{-1}\langle x,t\rangle}\D x$  and $\mathcal F^{-1}f(t)=\widehat f(-t)$, $t\in\mathbb{R}^n$, respectively. The translation of a function $f$ for $k\in\mathbb{Z}^n$ is denoted by $T_kf(x)=f(x-k)$.

\section{Preliminaries}
The paper \cite{aap} provides a detailed exposition of results that are obtained adapting results of \cite{MB} to Sobolev spaces.

A space $W$ is called shift-invariant (SI for short) if\ $f\in W$ implies $T_kf(\cdot)=f(\cdot-k)\in W$, for any $k\in \mathbb{Z}^n$. Sobolev space $H^s:=H^s(\mathbb{R}^n)$, $s\in\mathbb{R}$, consists of all $f\in\mathcal{S}^\prime(\mathbb{R}^n)$ for which $\|\widehat{f}\|_{L^2_s}<+\infty$. Note, $L_s^2=\mathcal{F}(H^s)$  and
 $\langle f,g\rangle_{H^s}=\int_{\mathbb{R}^n}\widehat{f}(t)\overline{\widehat{{g}}}(t)\mu_s^2(t)\D t$.
 The space $H\left(\mathbb{T}^n,\ell^2_{s}\right)$ of all vector-valued measurable functions $F:\mathbb{T}^n\to\ell_s^2$ such that $\int_{\mathbb{T}^n}\|F(t)\|_{\ell_s^2}^2\D t<+\infty$, with scalar product
 $\langle F_1,F_2\rangle_{H(\mathbb{T}^n,\ell_s^2)}=\int_{\mathbb{T}^n}\langle F_1(t),F_2(t)\rangle_{\ell_s^2}\D t$ is a Hilbert space. For $s=0,$ it is denoted as $L^2\left(\mathbb{T}^n,\ell^2\right).$

 The space $S_s(\mathcal{A}_{I,s})=\overline{\spn}(E_s(\mathcal{A}_{I,s}))$ is called a principal shift-invariant (PSI for short) space if $\mathcal{A}_{I,s}=\{\varphi\}$, $\varphi\in H^s$, and a finitely generated shift-invariant (FSI for short) space if $\mathcal{A}_{I,s}=\{\varphi_1, \ldots,\varphi_r\}$, $\varphi_i\in H^s$, $i=1,\ldots,r$.

For $f\in H^s$ and $g\in L^2$ such that $f=(1-\frac{\Delta}{4\pi^2})^{-s/2}g$, the diagram
  \begin{align*}&L^2\quad \xrightarrow{\mathcal{T}}\quad L^2(\mathbb{T}^n,\ell^2)\\
  &\downarrow \tau_s\hspace*{2cm}\downarrow \nu_s\\
  &H^s \quad\xrightarrow{\mathcal{T}_s} \quad H(\mathbb{T}^n,\ell_s^2),
  \end{align*} commutes, where $\tau_s\big(g\big)=f$ and $\nu_s\big(\big(\widehat{g}(\cdot+k)\big)_{k\in\mathbb{Z}^n}\big)=\big(\frac{\widehat{g}(\cdot+k)}{\mu_s(k)}\big)_{k\in\mathbb{Z}^n}$, for all $s\in\mathbb{R}$.

A mapping $J_s:\mathbb{T}^n\to\big\{\text{closed subspaces of }\,\ell_s^2\big\}$ is called a range function.
%  Next, $J_s$ is measurable if for each $a$, $b\in\ell_{s}^2$, the scalar function $t\mapsto\langle P_{J_s}(t)a,b\rangle_{\ell_s^2}$ is measurable, where $P_{J_s}(t):\ell_{s}^2\to J_s(t)$, $t\in\mathbb{T}^n$, are associated orthogonal projections.
The space $M_{J_s}=\{F\in H(\mathbb{T}^n,\ell_{s}^2):F(t)\in J_s(t)$ for a.e.\ $t\in\mathbb{T}^n\}$ is a closed subspace of $H(\mathbb{T}^n,\ell_s^2)$.
There is one-to-one correspondence between a closed SI space $V_s=\big\{\varphi\in H^s\,:\,\mathcal{T}_s \varphi(t)\in{J_s}(t) \text{ for a.e. }t\in\mathbb{T}^n\big\}$ and a measurable range function $J_s$. This idea first appears in Helson's book \cite{HH}.

 \begin{theorem} \label{V} For every closed SI space $V_s\subset H^s$ there exists a measurable range function $J_s$ such that $\mathcal{T}_sV_s=M_{J_s}$, and vice versa. In particular, for $V_s=S_s(\mathcal{A}_{I,s})$,
$J_s(t)=\overline{\spn}\{\mathcal{T}_s\varphi(t):\varphi\in\mathcal{A}_{I,s}\}$.
 \end{theorem}
% Some authors call $\mathcal{T}f(t)$, i.e.\ $\big(\widehat{f}(t+q)\big)_{q\in\mathbb{Z}^n}$ the fiber for function $f$ at $t$, and the space $J(t)$ the fiber spaces for $V$ at $t$ (see \cite{A1,A2}).

 %  Recall that $E_s(\mathcal{A}_{I,s})$ is a Bessel family for $S_s(\mathcal{A}_{I,s})$ with bound $B<+\infty$ if
 %   $$\sum_{(k,i)\in\mathbb{Z}^n\times I}|\langle f,T_k\varphi_i\rangle_{H^s}|^2\leq B\|f\|_{H^s}^2,\quad\quad f\in\spn(E_s(\mathcal{A}_{I,s})).$$
 %   A Bessel family $E_s(\mathcal{A}_{I,s})$ is called a frame for $S_s(\mathcal{A}_{I,s})$ with frame bounds $A$, $B$, if additionally there exists $A>0$ such %that
 %   $$A\|f\|_{H^s}^2\leq\sum_{(k,i)\in\mathbb{Z}^n\times I}|\langle f,T_k\varphi_i\rangle_{H^s}|^2,\quad\quad f\in\spn(E_s(\mathcal{A}_{I,s})).$$
 %   We call $E_s(\mathcal{A}_{I,s})\subset H^s$ a Riesz family with bounds $A,B>0$ if
 %   $$A\|c\|_{\ell^2_s}^2\leq\bigg\|\sum_{(k,i)\in\mathbb{Z}^n\times I}c_iT_k\varphi_i\bigg\|_{H^s}^2\leq B\|c\|_{\ell^2_s}^2,$$
 %   hold for all $($finitely supported$)$ sequences $c=(c_i)_{i\in I}$.

Some references for frame theory are \cite{OC} and \cite{KG}. Since the theory in the $L^2$-case is enough complex, we carefully analyze the range function $J_s$ acting on spaces $V_s=S_s(\mathcal{A}_{I,s}).$

Shifts of a given set of functions $\mathcal{A}_{I,s}$ form a frame (a Riesz family)
for $S_s(\mathcal{A}_{I,s})$ precisely when these functions form a frame (a Riesz family) with
uniform constants on the fibers over the base space $\mathbb{T}^n$ in the Fourier domain.
  \begin{theorem}[\cite{aap}]\label{frame} $E_s(\mathcal{A}_{I,s})$ is a frame or  a Riesz basis for $S_s(\mathcal{A}_{I,s})$ with bounds $A$, $B$ or a Bessel family with bound $B$ for every $s\in\mathbb{R}$ $($equivalently, for some $s\in\mathbb{R})$ if and only if $\{\mathcal{T}_s\varphi(t):\varphi\in\mathcal{A}_{I,s}\}\subset\ell_s^2$ is a frame or  a Riesz basis for $J_s(t)$ with bounds $A$, $B$ or a Bessel family with bound $B$ for a.e.\ $t\in\mathbb{T}^n$, for every $s\in\mathbb{R}$ $($equivalently for some $s\in\mathbb{R})$,  respectively. Moreover, $E_s(\mathcal{A}_{I,s})$ is a fundamental frame for every $s\in\mathbb{R}$ if and only if $\{\mathcal{T}_s\varphi(t):\varphi\in\mathcal{A}_{I,s}\}\subset\ell_s^2$ is a fundamental frame for a.e.\ $t\in\mathbb{T}^n$, for every $s\in\mathbb{R}$.
  \end{theorem}
  The dimension function $\dim_{V_s}:\mathbb{T}^n\to\mathbb{N}_0\cup\{+\infty\}$ of $V_s$ defined by $\dim_{V_s}(t)=\dim J_s(t)$, where $V_s=\mathcal{T}_s^{-1}M_{J_s}$ and the spectrum of $V_s$ given by $\sigma_{V_s}=\big\{t\in\mathbb{T}^n:J_s(t)\neq\{{\bf0}\}\big\}$ are the main tools for the decomposition theorem.
  We say that $\varphi_0$ is a quasi-orthogonal generator of $S_s(\varphi)$ if and only if $\|\mathcal{T}_s\varphi_0(t)\|_{\ell_s^2}=\mathbf{1}_{\sigma_{S_s(\varphi)}}(t)$ for a.e. $t\in\mathbb{T}^n$, (see \cite{aap} for more details).

De Boor et al. in \cite{BDR} proved the decomposition theorem for FSI space into quasi-regular spaces and Bownik in \cite{MB} gave the decomposition of a SI subspace of $L^2(\mathbb{R}^n)$ into an orthogonal sum of PSI spaces each of which is generated by a quasi-orthogonal generator. The method of proof carries over to domain of SI spaces of Sobolev type.

The proof of the following assertion will be given in a separate contribution, \cite{aap1}.
   \begin{theorem} \label{D} Every SI space $V_s\subset H^s$ can be decomposed as an orthogonal sum
   of PSI space $V^i_s$, $i\in\mathbb{N}$, with quasi-orthogonal generators $\varphi_i$, $i\in\mathbb{N}$,
    and $\sigma_{V_s^{i+1}}\subset\sigma_{V_s^i}$, for all $i\in\mathbb{N}$. Moreover,  $\dim_{V_s}(t)=\sum_{i\in\mathbb{N}}\|\mathcal{T}_s\varphi_i(t)\|_{\ell_s^2}$, for a.e. $t\in\mathbb{T}^n.$
   \end{theorem}
 As an application of the previous theorem, shift-preserving operators can characterized using range operators and some properties of the dimension function can be proved.

\section{Main results}

In \cite{MB} (based on \cite{BDR}, \cite{BDR2}, \cite{HH}, \cite{RS}), shift-preserving operators on $L^2(\mathbb{R}^n)$ are defined and their properties are derived. Also, the range operator on $L^2(\mathbb{R}^n)$ is introduced and important connections between range operators and shift-preserving operators are given. Aguilera and collaborators continued to investigate properties of the range function, range operators and shift-preserving operators \cite{ACCP1,ACCP2}.

 The definition, assertions and proofs are adapted to our setting and spaces.

A bounded linear operator $L:V_s\to H^s$ is shift-preserving if $LT_k=T_kL$ for all $k\in\mathbb{Z}^n$.  A range operator on $J_s$ (with values in $\ell_s^2$) is a mapping $R_s$ which maps $\mathbb{T}^n$ to a set of bounded operators defined on closed subspaces of $\ell_s^2$, such that the domain of $R_s(t)$ equals $J_s(t)$ for a.e.\ $t\in\mathbb{T}^n$.

The correspondence between the shift-preserving operator and the corresponding range operator is one-to-one under the convention that the range operators are identified if they are equal almost everywhere.

   \begin{theorem} For a shift-preserving operator $L:V_s\to H^s$ there exists a measurable range operator $R_s$ on $J_s$ so that
   \begin{equation}\label{3.6}(\mathcal{T}_s L)f(t)=R_s(t)(\mathcal{T}_s f(t))\quad\text{for a.e. }t\in\mathbb{T}^n,\, f\in V_s. \end{equation}
  Also, for a measurable range operator $R_s$ on $J_s$ such that \[\esssup\limits_{t\in\mathbb{T}^n}\|R_s(t)\|<+\infty,\] there exists a shift-preserving operator $L:V_s\to H^s$ that $\eqref{3.6}$ is satisfied. Moreover, $L$ is an isometry if and only if $R_s(t)$ is an isometry, for a.e.\ $t\in\mathbb{T}^n$.
   \end{theorem}

   \begin{theorem} Let $L:V_s\to V_s$ be a shift-preserving operator with corresponding range operator $R_s$ on a SI space $V_s\subset H^s$ with the associated range function $J_s$. The adjoint operator $L^*:V_s\to V_s$ with the corresponding range operator $R_s^*(t)=(R_s(t))^*$, for a.e.\ $t\in\mathbb{T}^n$, is also shift-preserving. Moreover, $L$ is self-adjoint and $\sigma(L)\subseteq[A,B]$ if and only if $R_s(t)$ is self-adjoint and $\sigma(R_s(t))\subseteq[A,B]$ for a.e.\ $t\in\mathbb{T}^n$, where $A$, $B\in\mathbb{R}$, $A\leqslant B$. Also, $L$ is unitary if and only if $R_s(t)$ is unitary operator for a.e.\ $t\in\mathbb{T}^n$.\end{theorem}

    \begin{proof} It is obvious that $R_s^*$ is measurable and uniformly bounded on $J_s$. There exists a shift-preserving operator $L^\bullet:V_s\to H^s$ such that $(\mathcal{T}_s L^\bullet)f(t)=R_s^*(t)(\mathcal{T}_s f(t))$, $f\in V_s$.
   For $f,g\in V_s$, and using the fact that $\mathcal{T}_s$ is an isometric isomorphism, we get
   \begin{align*}\langle Lf,g\rangle_{H^s}=\langle(\mathcal{T}_s L)f,\mathcal{T}_s g\rangle_{H(\mathbb{T}^n,\ell_s^2)}
   &=\int_{\mathbb{T}^n}\langle R_s(t)(\mathcal{T}_s f(t)),\mathcal{T}_s g(t)\rangle_{\ell_s^2}\D t\\
   &=\int_{\mathbb{T}^n}\langle\mathcal{T}_s f(t),R_s^*(t)(\mathcal{T}_s g(t))\rangle_{\ell_s^2}\D t\\
   &=\langle\mathcal{T}_s f,(\mathcal{T}_sL^\bullet)g\rangle_{H(\mathbb{T}^n,\ell_s^2)}=\langle f,L^\bullet g\rangle_{H^s}.\end{align*}
Thus,  $L^*=L$ if and only if $R_s^*(t)=R_s(t)$ for a.e. $t\in\mathbb{T}^n$.

  Let $\sigma(L)\subseteq[A,B]$, i.e. for all $f\in V_s$, we have
   \begin{align}\label{3.13}\nonumber A\|f\|_{H^s}^2\leq\langle Lf,f\rangle_{H^s}&=\langle (\mathcal{T}_sL)f,\mathcal{T}_sf\rangle_{H(\mathbb{T}^n,\ell_s^2)}\\
   &=\int_{\mathbb{T}^n}\langle R_s(t)(\mathcal{T}_s f(t)),\mathcal{T}_s f(t)\rangle_{\ell_s^2}\D t\leq B\|f\|_{H^s}^2,
   \end{align}
   and the assertion holds. Suppose, now, that $\sigma(R_s(t))\subseteq[A,B]$ for a.e.\ $t\in\mathbb{T}^n$, \ie \, for $f\in V_s$
   $$A\|\mathcal{T}_s f(t)\|_{\ell_s^2}^2\leq\big\langle R_s(t)(\mathcal{T}_s f(t)),\mathcal{T}_s f(t)\big\rangle_{\ell_s^2}\leq B\|\mathcal{T}_s f(t)\|_{\ell_s^2}^2,\quad \mbox{for a.e.}\,\, t\in\mathbb{T}^n.$$
   Integrating over $\mathbb{T}^n$, we obtain \eqref{3.13}.

 Using the previous facts, if $L$ is unitary, then $\sigma(LL^*)=\sigma(L^*L)=\{1\}$, if and only if $\sigma(R_s(t)R_s^*(t))=\sigma(R_s^*(t)R_s(t))=\{1\}$ for a.e.\ $t\in\mathbb{T}^n$, \ie\, $R_s(t)$ is unitary for a.e.\ $t\in\mathbb{T}^n$.
   \end{proof}

 By studying the shift-preserving operator defined on $V_s$, some properties of the dimension function can be proved.
   \begin{theorem} Let $L:V_s\to V_s$ be shift-preserving operator and $\widetilde{V_s}=\overline{L(V_s)}$. Then, $\dim_{\widetilde{V_s}}(t)\leqslant\dim_{V_s}(t)$ for a.e.\ $t\in\mathbb{T}^n$.
   \end{theorem}
Two SI spaces can be mapped onto each other with an iso-
morphism commuting with shifts precisely when they have identical dimension
functions almost everywhere.

Finally, a result about
dual frames is obtained using the range operator approach.

The frame operator $F:S_s(\mathcal{A}_{I,s})\to S_s(\mathcal{A}_{I,s})$ is given by
  $$Ff=\sum_{(k,i)\in\mathbb{Z}^n\times I}\langle f,T_k\varphi_i\rangle_{H^s}T_k\varphi_i,$$
where the convergence is unconditional in $H^s$.

Results about dual frames are obtained
in \cite{RS}, with the approach of the mixed Gramian. For
a fuller treatment, we refer the reader to \cite{D}.

%Bownik uses these obtained results and gives the properties of the dimension function and determines a dual frame for a given frame.

  \begin{theorem} For a frame $($a Riesz family$)$ $E_s(\mathcal{A}_{I,s})$ with constants $A$, $B$ there exists the dual frame $($dual Riesz family$)$ $E_s(\mathcal{B}_{I,s})$ with constants $B^{-1}$, $A^{-1}$, where $\mathcal{B}_{I,s}$ is given by $\mathcal{B}_{I,s}=\{\theta_i:\theta_i=F^{-1}\varphi_i, i\in I\}$. Moreover,
$\mathcal{T}_s\theta_i(t)=R_s^{-1}(t)\big(\mathcal{T}_s\varphi_i(t)\big)$ for a.e. $t\in\mathbb{T}^n$, $i\in I$.
  \end{theorem}
 \begin{proof} For $f\in S_s(\mathcal{A}_{I,s})$ it follow that
    \begin{align*}\sum_{(k,i)\in\mathbb{Z}^n\times I}\big|\langle f,F^{-1}T_k\varphi_i\rangle_{H^s}\big|^2
    &=\sum_{(k,i)\in\mathbb{Z}^n\times I}\big|\langle F^{-1}f,T_k\varphi_i\rangle_{H^s}\big|^2\\
    &=\sum_{(k,i)\in\mathbb{Z}^n\times I}\langle F^{-1}f,T_k\varphi_i\rangle_{H^s}\overline{\langle F^{-1}f,T_k\varphi_i\rangle}_{H^s}\\
    &=\sum_{(k,i)\in\mathbb{Z}^n\times I}\big\langle\langle F^{-1}f,T_k\varphi_i\rangle_{H^s}T_k\varphi_i,F^{-1}f\big\rangle_{H^s}\\
    &=\bigg\langle\sum_{(k,i)\in\mathbb{Z}^n\times I}\langle F^{-1}f,T_k\varphi_i\rangle_{H^s}T_k\varphi_i,F^{-1}f\bigg\rangle_{H^s}\\
    &=\big\langle F(F^{-1}f),F^{-1}f\big\rangle_{H^s}=\langle F^{-1}f,f\rangle_{H^s}.
    \end{align*} Thus, we get
    \begin{equation}\label{5.5}B^{-1}\|f\|_{H^s}\leq\sum_{(k,i)\in\mathbb{Z}^n\times I}\big|\langle f,F^{-1}T_k\varphi_i\rangle_{H^s}\big|^2
    =\langle F^{-1}f,f\rangle_{H^s}\leq A^{-1}\|f\|_{H^s}.
    \end{equation} Consequently, $\{F^{-1}T_k\varphi_i:k\in\mathbb{Z}^n,i\in I\}$ is a dual frame of $E_s(\mathcal{A}_{I,s})$ with constants $B^{-1}$, $A^{-1}$. Furthermore, $F$ is a shift-preserving operator. Thus, for $g=Ff\in H^s$,
    we have $$F^{-1}T_kg=F^{-1}T_kFf=F^{-1}FT_kf=T_kf=T_kF^{-1}g,\quad k\in\mathbb{Z}^n.$$
    Hence, $F^{-1}$ is shift-preserving. Now, it follows that $E_s(\mathcal{B}_{I,s})$ is a dual frame of $E_s(\mathcal{A}_{I,s})$ with constants $B^{-1}$, $A^{-1}$. Moreover,
    $$\sum_{(k,i)\in\mathbb{Z}^n\times I}\langle f,T_k\theta_i\rangle_{H^s} T_k\varphi_i=\sum_{(k,i)\in\mathbb{Z}^n\times I}\langle f,T_k\varphi_i\rangle_{H^s} T_k\theta_i=f\in S_s(\mathcal{A}_{I,s})$$
    holds with the unconditional convergence in $H^s$.
    \end{proof}
  \begin{remark} In the case of Riesz family, the biorthogonality relation
\[\langle T_k\varphi_i,T_\ell\theta_j\rangle_{H^s}=\delta_{k,\ell}\delta_{i,j}, \quad k,\ell\in\mathbb{Z}^n, \;i, j\in I,\] where $\delta_{i,j}=\begin{cases}1, & i=j,\\ 0, & i\neq j,\end{cases}$ is satisfied.
  \end{remark}

\section*{Acknowledgment}

This research was supported by the Science Fund of the Republic of Serbia,  $\#$GRANT No $2727$,{\it Global and local analysis of operators and distributions} - GOALS. The first
author is supported by the Ministry of Science, Technological Development and Innovation
of the Republic of Serbia Grant No. 451-03-66/2024-03/200122. The second author is supported by the Ministry of Science, Technological Development and
Innovation of the Republic of Serbia (Grants No. 451-03-65/2024-03/200122).

\end{document}